\documentclass[smallextended,referee,envcountsect]{svjour3}

\usepackage{float, enumerate, graphicx, amssymb, amsmath , color, longtable, multirow, lineno,url}
\usepackage{graphicx}
\usepackage[justification=centering]{caption}
\usepackage{amsmath}

\usepackage[active]{srcltx}
\usepackage{pifont}
\usepackage{paralist} 
\usepackage{amssymb}
\usepackage{authblk}
\usepackage{multirow,booktabs}
\usepackage{graphicx}
\graphicspath{{figures/}}
\usepackage{subfigure}
\usepackage{epsfig}
\usepackage{epstopdf}
\usepackage{longtable,booktabs}

\numberwithin{assumption}{section}
\numberwithin{equation}{section}

\usepackage{algorithm}
\usepackage{algorithmic}
\usepackage[a4paper,left=2cm,right=2cm,bottom=2.3cm,top=2.3cm]{geometry}
\usepackage[colorlinks,linkcolor=blue,citecolor=blue,                  bookmarksnumbered=true,
bookmarksopen=true,
colorlinks ]{hyperref}
\smartqed
\usepackage{graphicx}
\journalname{  }
\begin{document}
	
	\title{ A unified framework for inexact adaptive stepsizes in the gradient methods, the conjugate gradient methods  and the quasi-Newton methods  for strictly convex quadratic optimization   }
\author{ Zexian Liu   }
\institute{
 	Zexian Liu,   e-mail: liuzexian2008@163.com
	\at  School of Mathematics and Statistics, Guizhou University, Guiyang, 550025, China 
}
\date{Received: date / Accepted: date}
	
	\maketitle
	
	\begin{abstract} 
	The inexact adaptive stepsizes for the conjugate gradient method and  the quasi-Newton method are very rare. 		The exact stepsizes in the gradient method, the conjugate gradient method and the  quasi-Newton method for strictly convex quadratic optimization have a unified framework, while the unified framework for inexact adaptive stepsizes  in the gradient method, the conjugate gradient method and the quasi-Newton method for strictly convex quadratic optimization  still remains unknown.   Based on the above observations,   we propose a  unified framework for  inexact adaptive stepsizes in the gradient method, the conjugate gradient method  and the quasi-Newton method  for strictly convex quadratic optimization, which is called approximately optimal stepsize. The global convergence and the convergence rate of the gradient method with the approximately optimal stepsize are established by exploring  the relation between the approximately optimal stepsize and the famous Barzilai-Borwein (BB) stepsizes.   Some  numerical results are presented, which confirm   the remarkable  numerical advantage of the gradient method, the conjugate gradient method and the quasi-Newton method with the unified framework for  inexact adaptive stepsizes. Some open problems about the gradient method, the conjugate gradient method and the  quasi-Newton method with approximately optimal stepsize  are raised.
	\end{abstract}
	\section{Introduction} 
Consider the following strictly convex quadratic unconstrained optimization problem:
\begin{equation}\label{eq:QuadraticMin} \mathop {\min }\limits_{x \in {R^n}} f(x) = \frac{1}{2}{x^T}Ax - {b^T}x,  \end{equation} 
where $b\in R^n$, and $A \in R^{n \times n}$ is symmetric and positive definite. 

Given  an initial point $ x_0 $, the iterative method for solving problem \eqref{eq:QuadraticMin} takes the following form:  \begin{equation} \label{iterativeform}
x_{k+1} = x_k + \alpha_k d_k, \;\;\;k=0,1,2,\cdots
\end{equation} 
where $\alpha_k$ is the stepsize, and $ d_k $ is the search direction.

(i) If $d_k= -g_k 		$, then the iterative method corresponds to the gradient method.

(ii) If $d_k		$ is defined by
\[{d_k} = \left\{ \begin{array}{l}
- {g_k},\;\;\;\;\;\;\;\;\;\;\;\;\;\;\;\;\; \text{if}\;k = 0,\\
- {g_k} + {  \beta _k} {d_{k - 1}},\;\;\;\text{if}\;k > 0,
\end{array} \right.\]
where ${g_k} = \nabla f\left( {{x_k}} \right)$, $ \beta_k $ is the conjugate parameter, the iterative method corresponds to the  conjugate gradient method.    Some  well-known formulae for $ \beta_k $ are called the
Fletcher-Reeves (FR) \cite{Fletcher1964Function}, Hestenes-Stiefel (HS)   \cite{Hestenes1952Methods}, Polak-Ribi\`ere-Polyak (PRP)   \cite{Polyak1969The,Polak1969Note}   and Dai-Yuan (DY)  \cite{Dai1999A}
formulae, and are given by
\[\beta _k^{FR} = \frac{{{{\left\| {{g_{k}}} \right\|}^2}}}{{{{\left\| {{g_{k-1  }}} \right\|}^2}}},\;\;\;\;\beta _k^{HS} = \frac{{g_{k }^T{y_{k-1  }}}}{{d_{k-1  }^T{y_{k-1  }}}},\;\;\;\beta _k^{PRP} = \frac{{g_{k}^T{y_{k-1 }}}}{{{{\left\| {{g_{k-1 }}} \right\|}^2}}},\;\;\;\beta _k^{DY} = \frac{{{{\left\| {{g_{k}}} \right\|}^2}}}{{d_{k-1  }^T{y_{k-1  }}}},\]
where $y_{k-1}=g_k -g_{k-1}.$

(iii) If $d_k		$ is defined by $$d_k= - B_k^{ - 1}{g_k},$$  
where  $B_k$ is a symmetric and positive definite approximation to the Hessian matrix, then the iterative method corresponds to the quasi-Newton method. The Broyden family method \cite{Broyden1967QuasiNewtonMA}  takes the following update formula
\begin{align}\label{Broyden2}
B_{k+1}^{\theta} =  B_k + \frac{y_k y_k^\top}{s_k^\top y_k} - \frac{B_k s_k s_k^\top B_k}{s_k^\top B_k s_k} + \theta \, \omega_k \omega_k^\top,
\end{align}
where $ s_k= x_{k+1}-x_k$, and
\begin{equation}\label{omega}
\omega_k = \sqrt{s_k^\top B_k s_k} \left(\frac{y_k}{s_k^\top y_k} - \frac{B_k s_k}{s_k^\top B_k s_k} \right).\nonumber
\end{equation}
When \(\theta = 1\), the formula \eqref{Broyden2} reduces to the Davidon-Fletcher-Powell (DFP) formula \cite{Davidon1991SIOPT,FletcherPowell1963CompJ}; when \(\theta = 0\), it reduces to the Broyden-Fletcher-Goldfarb-Shanno (BFGS) formula \cite{Broyden1970IMA,Fletcher1970CompJ,Goldfarb1970MC,Shanno1970MC}.

The  stepsize is   the exact stepsize \cite{Cauchy1847} if it  satisfies  
\begin{equation}\label{exactstepsize}	
{\alpha _k} = \arg \mathop {\min }\limits_{\alpha  > 0} \;\;f({x_k} + \alpha {d_k}) =  - \frac{{g_k^T{d_k}}}{{d_k^TA{d_k}}}.\end{equation}     
Though the exact stepsize possesses  nice theoretical properties, it  generally requires expensive cost or is often difficult to obtain, which makes it impossible to be  applied widely in  practice. 

The inexact  stepsize of gradient method for convex quadratic optimization  has   been studied widely \cite{BB1988,Frassoldati2008,DaiYuanGM2005Analysis,AM2003DY,ABBDai2006,SDC2014Asmundis,Raydan1993On,DaiLiao2002Rlinear,Yuanstepsize2006,Xie2025Adaptive,Huang2019HDL,Raydan1997GBB,Serafino2018On,Friedlander1998Gradient}. However, research results about the inexact adaptive stepsizes of the conjugate gradient method  and the quasi-Newton method  are scarce. This may be attributed to the complexity of the search directions of the conjugate gradient methods  and the quasi-Newton methods, in contrast to the search direction $-g_k$ of the gradient method. 	The exact stepsizes in  gradient method, conjugate gradient method and  quasi-Newton method for strictly convex quadratic optimization have a unified framework  $-\frac{{g_k^T{d_k}}}{{d_k^TA{d_k}}}$, while the unified framework for inexact adaptive stepsizes  in  the gradient method, the conjugate gradient method and the quasi-Newton method  for strictly convex quadratic optimization  still remains unknown.  A question is naturally to be asked: 

\textit{Can we develop a unified framework for inexact adaptive stepsizes  in the gradient method, the conjugate gradient method and the quasi-Newton method  for strictly convex quadratic optimization? }

 It is interesting to study the above problem. To address it,   we propose a  unified framework for  inexact adaptive stepsizes in the gradient method, the conjugate gradient method and the quasi-Newton method  for strictly convex quadratic optimization, which is called approximately optimal stepsize. The global convergence and the convergence rate of the gradient method with the approximately optimal stepsize are established by exploring the relation  between the approximately optimal stepsize and the famous Barzilai-Borwein (BB) stepsizes \cite{BB1988}. Some  numerical results are presented, which confirm the remarkable  numerical advantage of the gradient method, the conjugate gradient method and the quasi-Newton method with the unified framework.

	The remainder of this  paper is organized as follow. In Section 2, we present  a  unified framework for  inexact adaptive stepsizes in the gradient method, the conjugate gradient method and the quasi-Newton method  for strictly convex quadratic optimization, which is called approximately optimal stepsize.  In Section 3 we establish   the global convergence and convergence rate of the gradient method with the approximately optimal stepsize  by exploring the relation  between the approximately optimal stepsize and the famous Barzilai-Borwein (BB) stepsizes.    Some numerical experiments are conducted in Section 4.    Conclusions and some open problems    are given in the last section.

\section{ The   unified framework for  inexact adaptive stepsizes ---approximately optimal stepsize} 

We describe the motivation as follows. 

(i)The unified framework for inexact adaptive stepsizes  in  the gradient method, the conjugate gradient method and the quasi-Newton method  for strictly convex quadratic optimization  still remains unknown. 

 (ii)The search directions of the  conjugate gradient methods  and quasi-Newton methods for strictly convex quadratic optimization possesses  nice theoretical property such as conjugacy under the exact line search, which makes  the two methods enjoy twice finite termination property. The theoretical merits of conjugacy and finite termination endow the conjugate gradient methods and the quasi-Newton methods with considerable potential for delivering efficient performance in both convex and nonconvex optimization. Therefore, it   is of great significance to study inexact adaptive stepsizes of the gradient method, the conjugate gradient method  and the quasi-Newton method, especially for a  unified framework for their  inexact adaptive stepsizes. 
 
 (iii)The exact stepsizes in the gradient methods, the conjugate gradient methods and the quasi-Newton methods for strictly convex quadratic optimization have a unified framework  $-\frac{{g_k^T{d_k}}}{{d_k^TA{d_k}}}$. To develop a      unified framework for their  inexact adaptive stepsizes, one should ignore the    specific forms of their different search direction  to some extent  and borrow the idea of the exact stepsize \eqref{exactstepsize} to construct the  unified framework for their  inexact adaptive stepsizes.

Based on the above observations, we propose a new type of inexact adaptive  stepsize for the gradient methods, the conjugate gradient methods and the quasi-Newton methods, and call it approximately optimal stepsize, which is defined as follows.

  \textbf{Approximately optimal stepsize.}	Suppose $f$ is continuously differentiable, $d_k$ is the search direction of the  gradient methods, the conjugate gradient methods  or the quasi-Newton methods, and $\phi_k(\alpha)$ is an approximation model of the line search function $f(x_k + \alpha d_k)$. If a positive number $\alpha_k^{ {AOS}}$ satisfies
\[{\alpha_k ^ {AOS} } = \arg \;\mathop {\min }\limits_{\alpha >0}  {\phi_k }(\alpha ),　\]	
then, 	$\alpha_k^{ {AOS}}$  is called approximately optimal stepsize of the corresponding optimization method related to  $\phi_k(\alpha)$. 

\noindent\textbf{Remark 1} If $\phi_k(\alpha) = f(x_k + \alpha d_k)$, then the approximately optimal stepsize ${\alpha_k ^ {AOS} }$ reduces to the exact   stepsize, namely,  the optimal   stepsize. This is   the reason why it is called   approximately optimal stepsize.

\noindent\textbf{Remark 2} If ${\phi _k}\left( \alpha  \right) = {f_k} + \alpha g_k^T{d_k} + \frac{1}{2}{\alpha ^2}d_k^T{\bar B_k}{d_k}$, where $$ {\bar B_k} = \frac{{{{\left\| {{y_{k - 1}}} \right\|}^2}}}{{s_{k - 1}^T{y_{k - 1}}}}I - \frac{{{{\left\| {{y_{k - 1}}} \right\|}^2}}}{{s_{k - 1}^T{y_{k - 1}}}}\frac{{{s_{k - 1}}s_{k - 1}^T}}{{{{\left\| {{s_{k - 1}}} \right\|}^2}}} + \frac{{{y_{k - 1}}y_{k - 1}^T}}{{s_{k - 1}^T{y_{k - 1}}}} ,$$ then 
\begin{equation}\label{AOS}
\alpha _k^{AOS} =  - 	\frac{{g_k^T{d_k}}}{{d_k^T{\bar B_k}{d_k}}}.
\end{equation}
Obviously, approximately optimal stepsize \eqref{AOS} is independent of any parameter of $f$ and thus is adaptive.  When $ d_k =-g_k $, the gradient method with \eqref{AOS} is called GM\_AOS; when   $ d_k =-g_k + \beta_k ^{DY}d_{k-1},$ the conjugate gradient (CG) method  with \eqref{AOS} is called CG\_AOS;  when   $ d_k =-B_k^{-1}g_k $ with BFGS formula---\eqref{Broyden2} with $\theta=0$,  BFGS method  with \eqref{AOS} is called BFGS\_AOS. It is noted that  $\beta_k ^{DY}$ is selected  in   CG\_AOS due to the nice theoretical  property of the DY method \cite{Dai1999A}, and BFGS formula is selected in BFGS\_AOS due to the nice theoretical  property and numerical effectiveness of the BFGS method.   Therefore, \textbf{approximately optimal stepsize \eqref{AOS} can be regarded as a unified framework for inexact adaptive stepsizes in the gradient methods,  the conjugate gradient methods and the quasi-Newton methods for strictly convex quadratic optimization.} 

\noindent\textbf{Remark 3} It follows from the definition of approximately optimal stepsize that different approximation model leads to different approximately optimal stepsize, and we can design stepsize by constructing  approximation models of line search function. 

\noindent\textbf{Remark 4} If $d_k =-g_k$, then the  approximately optimal stepsize is the stepsize proposed by Liu et al. \cite{LiuGMAOS2018} in 2018. Some results about approximately optimal stepsizes of gradient method can be seen in these paper \cite{LiuGMAOS2018,LiuGMAOS2024,LiuGMAOSOPTL2018,LiuGMAOS2018NM,LiuGMAOS2018JCAM,LiuGMAOS2022RAIRO,LiuGMAOS2019JOTA}.

\noindent\textbf{Remark 5} Although the   strictly convex quadratic optimization is considered in the paper, the definition of  the   approximately optimal stepsize is also applicable to those cases of convex optimization and general unconstrained optimization.

\section{Some theoretical results about the gradient method with approximately optimal stepsize (GM\_AOS)}

 By setting $ d_k =-g_k  $ in  \eqref{AOS},  we obtain the approximately optimal stepsize for gradient method, which is described as follows.
\begin{equation}\label{GMAOS}\alpha _k^{AOS} =- 	\frac{{g_k^T{d_k}}}{{d_k^T{\bar B_k}{d_k}}}= \frac{{{{\left\| {{g_k}} \right\|}^2}}}{{\frac{{{{\left\| {{y_{k - 1}}} \right\|}^2}}}{{s_{k - 1}^T{y_{k - 1}}}}\left( {{{\left\| {{g_k}} \right\|}^2} - \frac{{{{\left( {g_k^T{s_{k - 1}}} \right)}^2}}}{{{{\left\| {{s_{k - 1}}} \right\|}^2}}}} \right) + \frac{{{{\left( {g_k^T{y_{k - 1}}} \right)}^2}}}{{s_{k - 1}^T{y_{k - 1}}}}}}\end{equation}	

\begin{theorem} 
	Suppose that $f\left( x \right)  $ is given by \eqref{eq:QuadraticMin}. Then the approximately optimal stepsize \eqref{GMAOS} satisfies  the following relation:
	\[\frac{1}{2}\alpha _k^{B{B_2}} < \alpha _k^{AOS} < 2\alpha _k^{B{B_1}},\]	
	where  	 	$  \alpha _k^{B{B_1}} = {\left\| {{s_{k - 1}}} \right\|^2}/s_{k - 1}^T{y_{k - 1}}  $ and $  \alpha _k^{B{B_2}} = s_{k - 1}^T{y_{k - 1}}/{\left\| {{y_{k - 1}}} \right\|^2} $   \cite{BB1988}. 
\end{theorem} 

\begin{proof} By Lemma 3.1 in \cite{LiuGMAOS2024}, we can obtain the the smallest and largest eigenvalues    of the matrix $ \bar B_k $	in \eqref{GMAOS};
 
$$ \lambda_{\max} = \frac{1}{\alpha_k^{\mathrm{BB}_2}} + \sqrt{\frac{1}{\alpha_k^{\mathrm{BB}_2}}\left(\frac{1}{\alpha_k^{\mathrm{BB}_2}} - \frac{1}{\alpha_k^{\mathrm{BB}_1}}\right)}, \quad \lambda_{\min} = \frac{1}{\alpha_k^{\mathrm{BB}_2}} - \sqrt{\frac{1}{\alpha_k^{\mathrm{BB}_2}}\left(\frac{1}{\alpha_k^{\mathrm{BB}_2}} - \frac{1}{\alpha_k^{\mathrm{BB}_1}}\right)}.$$

It follows that 

$$ \frac{1}{\alpha_k^{\mathrm{BB}_2}} + \sqrt{\frac{1}{\alpha_k^{\mathrm{BB}_2}}\left(\frac{1}{\alpha_k^{\mathrm{BB}_2}} - \frac{1}{\alpha_k^{\mathrm{BB}_1}}\right)} < \frac{1}{\alpha_k^{\mathrm{BB}_2}} + \sqrt{\frac{1}{\alpha_k^{\mathrm{BB}_2}} \cdot \frac{1}{\alpha_k^{\mathrm{BB}_2}}} = 2\frac{1}{\alpha_k^{\mathrm{BB}_2}} , $$
which implies that 
$$ \lambda_{\max}(\bar B_k) < 2\frac{1}{\alpha_k^{\mathrm{BB}_2}}. $$

By the definition of $\lambda_{\min}$, we obtain that
\[\frac{1}{{\alpha _k^{B{B_2}}}} - \sqrt {\frac{1}{{\alpha _k^{B{B_2}}}}\left( {\frac{1}{{\alpha _k^{B{B_2}}}} - \frac{1}{{\alpha _k^{B{B_1}}}}} \right)}  > \frac{1}{{\alpha _k^{B{B_2}}}} - \sqrt {\frac{1}{{\alpha _k^{B{B_2}}}}\frac{1}{{\alpha _k^{B{B_2}}}} - \frac{1}{{\alpha _k^{B{B_1}}}}\frac{1}{{\alpha _k^{B{B_2}}}} + \frac{1}{4}{{\left( {\frac{1}{{\alpha _k^{B{B_1}}}}} \right)}^2}}  =  \frac{1}{{2\alpha _k^{B{B_1}}}},
\]

 whih yields that 

\[
\lambda_{\min}(\bar B_k) \geq \frac{1}{{2\alpha _k^{B{B_1}}}}.
\]

According to \eqref{GMAOS}, we have 	\[\frac{1}{2}\alpha _k^{B{B_2}} < \alpha _k^{AOS} < 2\alpha _k^{B{B_1}}.\]
The proof is completed. \qed

\end{proof}

By Theorem 3.1, the approximately optimal stepsize  \eqref{GMAOS} can be expressed as the convex combination of two famous BB stepsizes: \[\alpha _k^{AOS} = \frac{t}{2}\alpha _k^{B{B_2}} + 2t\alpha _k^{B{B_1}},\;\;\text{where}\;t \in \left( {0,1} \right). \]
Therefore, we can establish the global convergence and the R-linear convergence rate of GM\_AOS by the ways similar to that in \cite{Raydan1993On,DaiLiao2002Rlinear}. As a result, we only  present the theoretical results.

\begin{theorem}
	Suppose that $f\left( x \right)  $ is given by \eqref{eq:QuadraticMin}. Then,   the sequence $\left\lbrace x_k \right\rbrace $     generated by  GM\_AOS  converges globally to   the minimal  point $x^*$ .
	
\end{theorem}

\begin{theorem}
	Suppose that $f\left( x \right)  $ is given by \eqref{eq:QuadraticMin}. Then,   the sequence $\left\lbrace x_k \right\rbrace $     generated by  GM\_AOS converges R-linearly to   the minimal  point $x^*$ .
	
\end{theorem} 

\section{Numerical experiments}
 In the section, we conduct some numerical experiments to validate  the effectiveness of the gradient method with approximately optimal stepsize (GM$ \_ $AOS), the conjugate gradient method  with approximately optimal stepsize (CG$ \_ $AOS) and the BFGS  method  with approximately optimal stepsize (BFGS$ \_ $AOS). Since the numerical numerical results about GM$ \_ $AOS  can be found in those papers \cite{LiuGMAOS2018,LiuGMAOS2024,LiuGMAOSOPTL2018,LiuGMAOS2018NM,LiuGMAOS2018JCAM,LiuGMAOS2022RAIRO,LiuGMAOS2019JOTA}, we do not list them here and only   present these numerical results about CG$ \_ $AOS  and BFGS$ \_ $AOS. 
  
\subsection{Some numerical results about    CG$ \_ $AOS  }

The termination condition: ${\left\| {{g_k}} \right\|_\infty } < {10^{ - 6}}$;  the initial point: $ x_0=[1,1,\cdots, 1]^T   $; the test methods: BB， method \cite{BB1988}, HDL method \cite{Huang2019HDL} and CG$ \_ $AOS; the test problem is problem \eqref{eq:QuadraticMin} with the following $ A $ and $ b $; the experimental software: Matlab. 
 
 \textbf{Problem 1:} $ A =\left(a_{ij} \right)_{n \times n} $ is the diagonal   matrix with the main diagonal elements: 
 $  a_{11}=0.001, a_{22}=1, a_{33}=2, \cdots, a_{nn}=n-1   ; \;b=[0,0,\cdots, 0]^T.  $
 
 \vskip5mm
 \begin{center}
 	\small{{\small Table 1~~The number of iteration number on Problem 1 } \\[2mm]
 		\begin{tabular}{|c |c|c|c|c|}  \hline
 			       & n=100  & n=500  & n=1000  & n=5000 \\ \hline
 			BB   & 795    & 3611    &	5165   & 14221   \\ \hline
 			HDL   & 377    & 657    & 641   & 1783  \\ \hline   
 			CG$ \_ $AOS   & 291    & 397    &	553   &  861 \\ \hline      
 			
 		\end{tabular}}
 		\\
 	\end{center}

 	\textbf{Problem 2:} 									$$ A=D^T*D, \;\;\; \text{其中}\;\; D=100*\left( rand(n,n)-5\right),\;\;  b =100*(rand(n,1)-0.5).  $$

 	\vskip5mm
 	\begin{center}
 		\small{{\small Table 2~~The number of iteration number on Problem 2} \\[2mm]
 			\begin{tabular}{|c |c|c|c|c|c|c|c |c|c|c}  \hline
 			      & n=100    & n=200  & n=300  & n=100 & n=200  & n=300   & n=100 & n=200  & n=300 \\ \hline
 				BB      & 19946    & 16611    & $ > $50000     & $ > $50000    & 9777 & 20385  &  4204   &6305 &$ > $50000 \\ \hline
 				HDL    & 6279    & 1836    & $ > $50000   &  2108   &   3693  & 5927   &     2034    &2565 &$ > $50000  \\ \hline   
 				CG$ \_ $AOS       & 4392    & 3048    &	  12147   &  6875    & 4107   & 5748    &    2631   &2608  &12535  \\ \hline      

 			\end{tabular}}
 			\\
 		\end{center}
 		
 	In Table 2, the first three columns ($ n=100,200,300 $) corresponds to the random seed 2, the second  three columns corresponds to  the random seed 5  and the third three columns corresponds to  the random seed 8.

 		\textbf{Problem 3: }  $ A =\left(a_{ij} \right)_{n \times n} $ is   the diagonal   matrix with the following main diagonal elements:  
 		$  a_{11} = \kappa (A) =10^5, \;\; a_{nn} = 1,\;\; \left[a_{22}, a_{33}, \cdots, a_{(n-1)(n-1)}\right]^T= a_{nn}+\left(a_{11}-a_{nn} \right) * rand(n-2,1),    b =\left[0,0, \dots 0 \right]^T \in R^{n}$.  
 		
 		\vskip5mm
 		\begin{center}
 			\small{{\small Table 3~~The number of iteration number on Problem 3 } \\[2mm]
 				\begin{tabular}{|c |c|c|c|c|c|c|c |c|c }  \hline
 					   & n=100    & n=500  & n=1000  & n=5000  & n=10000 \\ \hline
 					BB   & 4785    & 5596    &3841     &6024    & 5057 \\ \hline
 					HDL   & 192    & 624    &725   &  1034   &   2290    \\ \hline   
 					CG$ \_ $AOS &   564    & 766    &703   &   859    & 1214     \\ \hline      

 				\end{tabular}}
 				\\
 			\end{center}
 			
As shown in Table  1, we can see that CG$ \_ $AOS demonstrates remarkable numerical advantages over the BB method and the HDL method for Problem 1. Table 2 shows that CG\_AOS is superior to the BB method and outperforms the HDL method for Problem 2 when  $n$ is large. The same conclusion holds for Problem 3. The above three tables show  that  CG$ \_ $AOS is very efficient.
 			
%
%
%
%

\subsection{Some numerical results about BFGS$ \_ $AOS }  
    The termination condition: ${\left\| {{g_k}} \right\|_\infty } < {10^{ - 6}}$;  the initial point: $ x_0=[1,1,\cdots, 1]^T   $;  the test problem is problem \eqref{eq:QuadraticMin} with   $ A =\left(a_{ij} \right)_{n \times n} $ is the diagonal   matrix with the main diagonal elements 
    $ a_{11}=0.001, a_{22}=1, a_{33}=2, \cdots, a_{nn}=n-1   ;   b=[0,0,\cdots, 0]^T;  $   the test method is BFGS\_1 (BFGS method with the unit  stepsize 1)  and   BFGS\_AOS (BFGS method with the approximately optimal stepsize \eqref{AOS}). It is noted that the convergence of  BFGS\_1 is established under the condition that ${B_0} = LC$, where $C$ satisfies $\mu C \preceq    A \preceq    LC $ \cite{Nesterov2022Rates}. Here $\mu$ is the strong convexity parameter, and $L$ is the Lipschitz constant of the gradient. It means that if this condition is not satisfied, then BFGS\_1 may not converge.

%
%
%

		\begin{center}
			\small{{\small Table 4~~The number of iteration number   } \\[2mm]
				\begin{tabular}{|c|c|c |c|c|c|c|c|c|c |c|c|c|c }  \hline
					& $\begin{array}{l}
					{ {n = 100,}}\\
					{ {B_0(1)}}
					\end{array}$ & $\begin{array}{l}
					{ {n = 100,}}\\
					{ {B_0(2)}}
					\end{array}$ & $\begin{array}{l}
					{ {n = 100,}}\\
					{ {B_0(3)}}
					\end{array}$    & $\begin{array}{l}
					{ {n = 500,}}\\
					{ {B_0(1)}}
					\end{array}$  & $\begin{array}{l}
					{ {n = 500,}}\\
					{ {B_0(2)}}
					\end{array}$   & $\begin{array}{l}
					{ {n = 500,}}\\
					{ {B_0(3)}}
					\end{array}$  & $\begin{array}{l}
					{ {n = 1000,}}\\
					{ {B_0(1)}}
					\end{array}$   & $\begin{array}{l}
					{ {n = 1000,}}\\
					{ {B_0(2)}}
					\end{array}$   & $\begin{array}{l}
					{ {n = 1000,}}\\
					{ {B_0(3)}}
					\end{array}$    \\ \hline
					BFGS$ \_ $1  & F     & 225   & 322  & F      &F      &372    &  F   &  F   & 374\\ \hline
					    
					BFGS$ \_ $AOS & 108 &120 & 213       & 506    &471   &   334    &  834   & 701  &293  \\ \hline      

				\end{tabular}}
				\\
			\end{center}

 The numerical results are reported in Table 4. In Table 4, the initial matrix $ B_0(1)=1000*I$,  $ B_0(2)= I$,   $ B_0(3)= 0.001*I$; ``F'' represents that BFGS\_1 fails to solve the problem due to  NaN in the iterative process. As shown in Table  4, we can see that BFGS\_AOS can solve successfully more problems than BFGS\_1, and is superior to BFGS\_1. Therefore, BFGS\_AOS is very efficient.

\section{Conclusions and some open problems}	
	In the paper we present a unified framework for inexact adaptive stepsizes in the gradient methods, the conjugate gradient methods  and the  quasi-Newton methods for strictly convex quadratic optimization, which is called approximately optimal stepsize.   Some numerical results demonstrate that approximately optimal stepsize is useful and efficient.
	
	We have established the global convergence and convergence rate of GM\_AOS. Since the approximately optimal stepsize does not possess  the property $g_{k+1}^Td_k =0$ generally, it is   difficult to establish the   global convergence and convergence rate of CG\_AOS and BFGS\_AOS as the existing theoretical framework of conjugate gradient method and quasi-Newton method with the exact stepsize for strictly convex quadratic optimization can not be used. 
As for the unified framework for inexact adaptive stepsizes in the gradient methods, the conjugate gradient method and the quasi-Newton method for        strictly convex quadratic optimization,	we raise five open problems as follows:
	
%
%
	
	(1)Can the conjugate gradient method with approximately optimal stepsize (CG\_AOS)   converge theoretically  ?  If yes, what is the convergence rate?
	
 	(2)Can the BFGS method with approximately optimal stepsize (BFGS\_AOS)   converge theoretically  ?  If yes, what is the convergence rate ? 
	
	(3)Can the  gradient  methods  with approximately optimal stepsize    be extended from  strictly convex optimization to general convex optimization ?  
	
		(4)Can  the conjugate gradient methods with approximately optimal stepsize    be extended from strictly convex quadratic optimization to general convex optimization ? 
		
	(5)Can the quasi-Newton methods with approximately optimal stepsize be extended from strictly convex quadratic optimization to general convex optimization ?
	
 The above open problems is important, and we  will focus on these problems in the	future.
					
					\begin{acknowledgements}  I would like to thank Professor  Yu-Hong Dai in Chinese Academy of Sciences for his valuable and useful suggestions.   This research is supported by National Science Foundation of China (No.12261019).  
					\end{acknowledgements}

\noindent\textbf{Data availability.}					
	The datasets generated during and/or analysed during the current study are available from the corresponding author on reasonable request.	
	
\noindent\textbf{Conflict of interest.	}
	The authors declare no competing interests.

				\end{document}